\documentclass[a4paper,12pt]{amsart}

\usepackage{enumerate}
\usepackage{esint}
\usepackage[latin1]{inputenc}
\usepackage{amsmath}
\usepackage{amsthm}
\usepackage{latexsym}
\usepackage{amssymb}
\usepackage{amsmath}
\usepackage{comment}
\usepackage{calc}
\usepackage{multicol}
\usepackage{slashed}
\usepackage{graphicx}
\usepackage{bbm}
\newtheorem{thm}{Theorem}[section]
\newtheorem{prop}{Proposition}[section]

\newtheorem{lem}{Lemma}[section]
\newtheorem{cor}{Corollary}[section]

\newtheorem{rem}{Remark}[section]

\newcommand{\R}{\mathbb{R}}
\numberwithin{equation}{section}

\newcommand{\vertiii}[1]{{\left\vert\kern-0.25ex\left\vert\kern-0.25ex\left\vert #1
\right\vert\kern-0.25ex\right\vert\kern-0.25ex\right\vert}}

\usepackage{xcolor}

\makeatletter
\newcommand{\leqnomode}{\tagsleft@true}
\newcommand{\reqnomode}{\tagsleft@false}
\makeatother

\thanks{
The  authors thank N. Visciglia and L. Forcella for the reading of a preliminary version of the paper.  J.B. and V.G. were partially supported by ``Problemi stazionari e di evoluzione nelle equazioni di campo non-lineari dispersive'' of GNAMPA 2020 and  by the project PRIN  2020XB3EFL by the Italian Ministry of Universities and Research.
V.G. was partially supported by  by   the Top Global University Project, Waseda University, by the University of Pisa, Project PRA 2018 49 and by Institute of Mathematics and Informatics, Bulgarian Academy of Sciences.
}

\begin{document}
\title[Compact embeddings  for fractional super and sub harmonic functions with radial symmetry]{Compact embeddings  for fractional super and sub harmonic functions with radial symmetry}
\author{Jacopo Bellazzini}
   \address{Jacopo Bellazzini \newline Dipartimento di Matematica  \\ Universit\`a di Pisa \\ Largo B. Pontecorvo 5, 56100 Pisa, Italy}%
\author{Vladimir Georgiev}
\address{V. Georgiev
\newline Dipartimento di Matematica Universit\`a di Pisa
Largo B. Pontecorvo 5, 56100 Pisa, Italy\\
 and \\
 Faculty of Science and Engineering \\ Waseda University \\
 3-4-1, Okubo, Shinjuku-ku, Tokyo 169-8555 \\
Japan and Institute of Mathematics and Informatics, Bulgarian Academy of Sciences, Acad.
Georgi Bonchev Str., Block 8, 1113 Sofia, Bulgaria}%
\maketitle
\begin{abstract}
We prove compactness of the embeddings in Sobolev spaces for fractional super and sub harmonic functions with radial symmetry.  The main tool is a  pointwise decay for
 radially symmetric functions belonging to a function space defined by finite homogeneous Sobolev norm together with finite $L^2$ norm of the Riesz potentials.
As a byproduct we prove also existence of maximizers for the interpolation inequalities in  Sobolev spaces for radially symmetric  fractional super and sub harmonic functions.
\end{abstract}

\section{Introduction}

The classical embedding in Sobolev spaces  $H^{S}(\R^d)\subset  \dot H^{r}(\R^d)$ for $0\leq r\leq S$ follows from
the  interpolation inequality  in homogeneous Sobolev spaces
\begin{equation}\label{GagliNiren}
\|D^r \varphi\|_{L^p(\R^d)}\leq C(r,S,p, d) \, \|\varphi\|_{L^2(\R^d)}^{1-\theta} \, \|D^{S} \varphi\|_{L^2(\R^d)}^{\theta}, \,,
\end{equation}
where  $\varphi\in H^{S}(\R^d)$ and $D^s \varphi$ is defined  by
$$
  (\widehat {D^s  \varphi}) (\xi)
  = |\xi|^s \widehat{\varphi} (\xi).
$$

The inequality \eqref{GagliNiren} holds, see  \cite{BM01}, Corollary 1.5 in \cite{HMOW11}, \cite{BM19} or  Theorem 2.44 in \cite{BCD} provided that
\begin{itemize}
\item $\frac{1}{p}=\frac{1}{2}+ \frac{r-\theta S}{d},$
\item $\frac{r}{S}\leq \theta\leq 1$,
\item $0<r\leq S$, $p>1$.
\end{itemize}

We notice that at the endpoint case $p=2$, corresponding to $\theta=\frac{r}{S}$, we have
\begin{equation}\label{GagliNirenq4}
\|D^r \varphi\|_{L^2(\R^d)}\leq C(r,S,2, d)\|\varphi\|_{L^2(\R^d)}^{1-\frac{r}{S}} \, \|D^{S} \varphi\|_{L^2(\R^d)}^{\frac{r}{S}},
\qquad \forall \varphi\in  H^{S}(\R^d),
\end{equation}
and hence the embedding  $H^{S}\subset  \dot H^{r}$ for $0\leq r\leq S$ is just a  consequence of \eqref{GagliNirenq4}.
If we look at the endpoint cases $\theta=\frac{r}{S}$ and $\theta=1$ in \eqref{GagliNiren} we obtain  that the range of  exponents $p$ without any symmetry and positivity assumption fulfills
\begin{equation}\label{eq:rangep}
\begin{aligned}
& p\in[2,\frac{2d}{d-2(S-r)}] &  \text{if } S-r<\frac{d}{2},\\
& p\in[2,\infty)  &  \text{if } S-r\geq \frac{d}{2}.
\end{aligned}
\end{equation}
We  remark that the lower endpoint does not depend on dimension $d$.

Moreover, looking at \eqref{GagliNirenq4}, it is easy to prove that the best constant in \eqref{GagliNirenq4} is $C(r,S,2, d)=1$. Indeed from H\"older's inequality in frequency applied to l.h.s. of \eqref{GagliNirenq4} we get $C(r,S,2, d)\leq 1$ and calling $A_n=\left\{\xi \in \R^d \text{ s.t. }1-\frac{1}{n}<|\xi|<1+\frac{1}{n}\right\}$ it suffices to consider
a sequence $\varphi_n$ such that $\hat \varphi_n(\xi)=\mathbbm{1}_{A_n}(\xi)$ to prove that $C(r,S,2, d)=1$.

 In the sequel we consider $ r,S,d$ as fixed quantities and we aim to study the range of $p$ such that  \eqref{GagliNiren} holds in case we restrict to  \emph{radially symmetric} functions $\varphi$ in $H^{S}(\R^d)$ such that $D^r \varphi$ is not only radially symmetric but also either \emph{positive} or \emph{negative}.

We introduce the notation for $0<r<s$
\begin{equation}
\dot H^{s}_{rad}(\R^d):=\{ \varphi \in \dot H^s(\R^d), \ \   \varphi=\varphi(|x|) \},
\end{equation}
\begin{equation}
H^{s}_{rad}(\R^d):=\{ \varphi \in H^s(\R^d), \ \   \varphi=\varphi(|x|) \},
\end{equation}
\begin{equation}
H^{s,r}_{rad, +}(\R^d):=\{ \varphi \in H^s_{rad}(\R^d), \ \  \ D^r \varphi \geq 0 \},
\end{equation}
\begin{equation}
H^{s,r}_{rad, -}(\R^d):=\{ \varphi \in H^s_{rad}(\R^d), \ \  \ D^r \varphi \leq 0 \}.
\end{equation}
By the relation $  (\widehat {-\Delta \varphi}) (\xi)
  = 4\pi^2 |\xi|^2 \widehat{\varphi} (\xi)= 4\pi^2(\widehat {D^2  \varphi}) (\xi)$ we shall emphasize that $H^{s,2}_{rad, +}(\R^d)$ corresponds to the set of \emph{superharmonic} radially symmetric functions belonging to $H^s(\R^d)$ while
$H^{s,2}_{rad, -}(\R^d)$ corresponds to the set of \emph{subharmonic} radially symmetric functions belonging to $H^s(\R^d)$. In the sequel we will call when $r\neq 2$
\emph{fractional superharmonic} radially symmetric functions belonging to $H^s(\R^d)$ the functions belonging to $H^{s,r}_{rad, +}(\R^d)$ and \emph{fractional subharmonic} radially symmetric functions belonging to $H^s(\R^d)$ the functions belonging to $H^{s,r}_{rad, -}(\R^d)$.

The main questions we are interesting in are the following ones:

Question A: Can we find appropriate values of $(r,S)$ such that $p$ can be chosen below $2$ in \eqref{GagliNiren} for fractional superharmonic (resp. subharmonic) functions belonging to $H^{S,r}_{rad, +}(\R^d)$?

\vspace{0.4cm}

Question B: If the answer of question A is positive, then can we expect a compact embedding of type
\begin{equation}\label{eq.cce1}
    H^{S,r}_{rad, +}(\R^d) \subset \subset \dot H^r(\R^d)?
\end{equation}

In the sequel we will consider the case  $\varphi \in H^{S,r}_{rad, +}(\R^d)$ but all the results are still valid if we consider $\varphi \in H^{S,r}_{rad, -}(\R^d)$.
The first result of the paper gives a positive answer to Question A.
\begin{thm}\label{thm:super}
Let $d\geq 2$ and $\frac 12 <r<\min(\frac{d}{2}, S-\frac 12)$,  then

\begin{equation}\label{GagliNiren5}
\begin{aligned}
&\|D^r \varphi\|_{L^p(\R^d)}\leq C_{rad,+}(r,S,p, d) \, \|\varphi\|_{L^2(\R^d)}^{1-\theta} \, \|D^{S} \varphi\|_{L^2(\R^d)}^{\theta}, \\
&  \forall \varphi  \ \in H^{S,r}_{rad,+}(\R^d) \,,
\end{aligned}
\end{equation}
with
\begin{align}
& p\in(p_0,\frac{2d}{d-2(S-r)}] & & \text{if } S-r<\frac{d}{2},\\
& p\in(p_0,\infty)  & & \text{if } S-r\geq \frac{d}{2},
 \end{align}
with $\theta$ fixed by the scaling equation
$$\frac{1}{p}=\frac{1}{2}+ \frac{r-\theta S}{d},$$ and $p_0<2$ is given by
$$p_{0}=\frac{d-2r+2(S-r)(d-1)}{-((S -r)-\frac 12)(d-2r)+2(S-r)(d-1)}.$$
\end{thm}

\begin{rem}
Theorem \ref{thm:super} holds also for  $\varphi  \ \in H^{S,r}_{rad,-}(\R^d)$. The crucial condition is that $D^r \varphi$ does not change sign.
\end{rem}

The constant  $C_{rad,+}(r,S,p, d)$ in \eqref{GagliNiren5} is defined as best constant in case of functions belonging to  $ H^{S,r}_{rad,+}(\R^d)$.

The fact that $p_0<2$ in the above Theorem implies  $D^r \varphi \in L^{p}$ with $p \in (p_0,2)$  and this allows us to obtain also a positive answer to Question B.
\begin{thm}\label{thm:main2}
Let $d\geq 2$ and $\frac 12 <r_0<\min(\frac{d}{2}, S-\frac 12)$, then the embedding
$$H^{S,r_0}_{rad,+}(\R^d)\subset \subset \dot H^{r}_{rad}(\R^d)
$$
is compact for any $0<r<S.$
\end{thm}
\begin{rem}
Theorem \ref{thm:main2} holds also  in $H^{S,r_0}_{rad,-}(\R^d)$. Clearly the main difficult in Theorem \ref{thm:main2} is to prove that the embedding $H^{S,r_0}_{rad,+}(\R^d)\subset \subset \dot H^{r_0}_{rad}(\R^d)$ is compact, the compactness for $r\neq r_0$
will follow by interpolation.
\end{rem}

As a second byproduct we have also the following result concerning the existence of maximizers for the interpolation inequality \eqref{GagliNiren5} in case $p=2$.
\begin{thm}\label{thm:main3}
Let $d\geq 2$ and $\frac 12 <r<\min(\frac{d}{2}, S-\frac 12)$  then
$$\|D^r \varphi\|_{L^2(\R^d)}\leq C_{rad,+}(r,S,2, d)\|\varphi\|_{L^2(\R^d)}^{1-\frac{r}{S}} \, \|D^{S} \varphi\|_{L^2(\R^d)}^{\frac{r}{S}}, $$
$$ \forall \varphi  \ \in H^{S,r}_{rad,+}(\R^d),$$
and the best constant $C_{rad,+}(r,S,2, d)$ is attained and  $C_{rad,+}(r,S,2, d)<1.$
\end{thm}

The strategy to prove  Theorem \ref{thm:super} and as a byproduct,  the compactness result given in Theorem \ref{thm:main2}, it to rewrite  \eqref{GagliNiren} involving $L^2$ norms of Riesz potentials when $0<r<d$.  By defining $u=D^r \varphi$ we obtain
\begin{equation}\label{GagliNiren2}
\| u\|_{L^p(\R^d)}\leq C(\alpha,s,p, d) \, \|\frac{1}{|x|^{\alpha}}\star u\|_{L^2(\R^d)}^{1-\theta} \, \|D^s u\|_{L^2(\R^d)}^{\theta}
\end{equation}
where $\alpha=d-r$, $s=S-r$. With respect to the new variables $\alpha, s$ we get without any symmetry or positivity assumption
\begin{equation}\label{eq:rangepm}
\begin{aligned}
& p\in[2, \frac{2d}{d-2s}] & & \text{if } s<\frac{d}{2},\\
& p\in[2, \infty)  & & \text{if }  s \geq \frac{d}{2}.
\end{aligned}
\end{equation}

If one considers functions fulfilling $D^r \varphi=u\geq 0$, inequality \eqref{GagliNiren2} is hence equivalent to the following inequality
\begin{equation}\label{GagliNiren3}
\| u\|_{L^p(\R^d)}\leq C(\alpha,s,p, d) \, \|\frac{1}{|x|^{\alpha}}\star |u| \|_{L^2(\R^d)}^{1-\theta} \, \|D^s u\|_{L^2(\R^d)}^{\theta}
\end{equation}
considering $|u|$ instead of $u$ in the Riesz potential. The strategy is hence  to prove that \emph{the radial symmetry} increases the range of $p$ for which
\eqref{GagliNiren3} holds and therefore as byproduct the  range of $p$ for which
\eqref{GagliNiren2} holds when $D^r \varphi=u$ is \emph{positive and radially symmetric} (resp. \emph{negative}). In particular we will show that the lower endpoint is allowed to be  below $p=2$.  A reasonable idea  to prove that the lower endpoint exponent  in \eqref{GagliNiren3} decreases with radial symmetry is to look at a suitable pointwise decay in the spirit of  the Strauss lemma \cite{S77} (see also   \cite{SS00,SS12} for  Besov and Lizorkin-Triebel classes). In our context where two terms are present, the Sobolev norm and  the  Riesz potential involving  $|u|$,   we have been inspired by  \cite{MVS} where the  case $s=1$ in \eqref{GagliNiren3} has been studied (see also \cite{BGO} and \cite{BGMMV}). For our purposes the fact that $s$ is in general not integer makes however the strategy completetly different from the one in \cite{MVS} and we need to estimate the decay of the high/low frequency part of the function to compute the decay. To this aim we compute the high frequency part using the explicit formula for the Fourier transform for radially symmetric function involving Bessel functions, in the spirit of \cite{CO}, while we use a  weighted $L^1$ norm to compute the decay for the low frequency part. The importance of a pointwise decay for the low frequency part involving weighted $L^p$ norms goes back to \cite{D} and we need to adapt it to our case in order to involve the Riesz potential.
Here is the step where \emph{positivity} is crucial. Indeed if one is interested to show a scaling invariant weighted inequality as
\begin{equation}\label{eq:scalinv}
\int_{\mathbb{R}^d} \frac{|u(x)|}{|x|^{\gamma}}dx \leq C \| \frac{1}{|x|^{\alpha}}\star |u| \| _{L^2(\R^d)}
\end{equation}
a scaling argument forces the exponent $\gamma$ to verify the relation $\gamma=\alpha-\frac{d}{2}$. Unfortunately \eqref{eq:scalinv} cannot hold in the whole Euclidean space following a general argument that goes back to  \cite{MVS} and \cite{R}. However a scaling invariant inequality  like \eqref{eq:scalinv} restricted on balls and on complementary of balls is enough for our purposes. Eventually, using all these tools, we are able to compute a pointwise decay that allows the lower endpoint for \eqref{GagliNiren3} to be below the threshold $p=2$.
Computed the pointwise decay we will follow the argument in \cite{BGO} to estimate the lower endpoint for  fractional superharmonic (resp. subharmonic) radially symmetric functions.

 Concerning the compactness we prove that taking a bounded sequence $\varphi_n \in H^{S,r}_{rad, +}$ then $\varphi_n \to \varphi$ $\dot{H}^r$ with $r>0$. Our strategy is to prove the smallness of $\|D^r(\varphi_n - \varphi)\|_{L^2(B_\rho)}$ and $\|D^r(\varphi_n - \varphi)\|_{L^2(B_\rho^c)}$ for suitable choice of the ball $B_\rho.$ For the first term we use Rellich-Kondrachov argument combined with commutator estimates, while for the exterior domain we use the crucial fact that $D^r(\varphi_n - \varphi)$ is in $L^p(|x|>\rho) $ for some $ p \in (1,2).$

Looking at the case $r=0$, by Rellich-Kondrachov we have $\|\varphi_n - \varphi\|_{L^2(B_\rho)} =o(1)$, however we can not obtain the smallness in the complementary $B_\rho^c$ of the ball so the requirement $r>0$ seems to be optimal.

It is interesting to look at the lower endpoint exponent $p_0$ given in Theorem \ref{thm:super} in case we consider radially symmetric superharmonic  (or subharmonic), namely when  $r=2$. In this case the condition  $\frac 12 <r<\min(\frac{d}{2}, S-\frac 12)$,  imposes to consider the case $d\geq 5$ and $S>\frac{5}{2}$. As an example we show on Figure \ref{G1} the graph of the function $p_0(S)$, that now is only a function of $S$, in lowest dimensional case $d=5$ that is a branch of hyperbola with asymptote $p_\infty=\lim_{S\rightarrow \infty} p_0(S)=8/7.$ It is interesting how the  regularity improves the  lower endpoint $p_0(S)$.

As a final comment we notice that for $d\geq 2$ if $D^2\varphi\geq 0$ then  $D^{\frac{3}{4}}\varphi=D^{-{\frac{5}{4}}}\left(D^2\ \varphi \right)\geq 0$ then,  taking $r_0=3/4$ and using the positivity of the Riesz kernel of  $D^{-{\frac{5}{4}}},$ we apply Theorem \ref{thm:main2} and we get the following corollary.
\begin{cor}
Let $\varphi_n$ be a sequence of radially symmetric superharmonic functions uniformly bounded in $H^2(\R^d)$, $d\geq 2$. Then for any $0<r<2$, up to subsequence
$\varphi_n \to \varphi$ in $\dot H^r(\R^d)$.
\end{cor}
\begin{figure}
  \includegraphics[width=10cm]{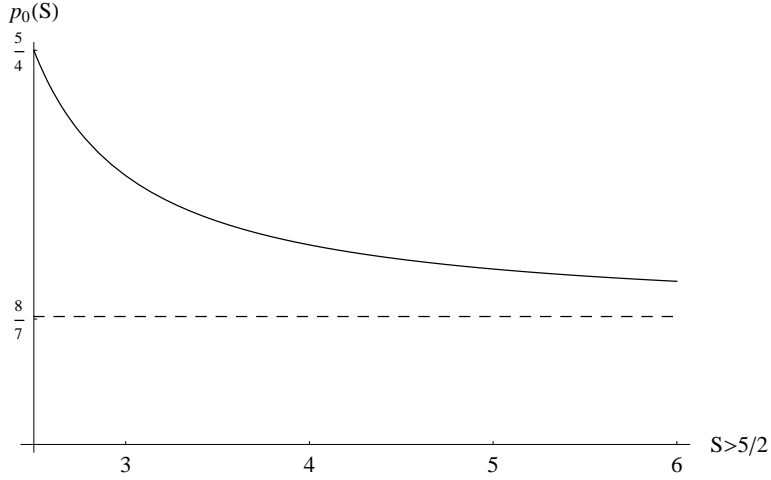}\\
  \caption{The graph of the function $p_0(S) = (16 S-30)/(14 S-27)$ in the case of superharmonic or subharmonic functions. Here $r=2,d=5$ and $S>5/2.$}\label{G1}
\end{figure}

\section{Interpolation inequalities for radial functions involving  Riesz potentials.}

Let $d\geq 2$, $0<\alpha<d$,   $\frac 12 <s,$ we define
$$X=X_{s,\alpha,d}=\left\{u \in \dot H^s_{rad }(\R^d), \ \ \left\|\frac{1}{|x|^{\alpha}}\star |u|\right\|_{L^2}<+\infty \ \right\}.$$

The aim of this section is to prove the following
\begin{thm}\label{thm:main}
Let $u\in X$ with $d\geq 2$, $s>\frac 12$, $\frac{d}{2}<\alpha<d-\frac 12$,  then $u \in L^p(\R^d)$ with
\begin{align*}
& p\in(p_{rad}, \frac{2d}{d-2s}] & & \text{if } s<\frac{d}{2},\\
& p\in(p_{rad}, \infty)  & & \text{if }  s \geq \frac{d}{2}.
\end{align*}
where $p_{rad}<2$ with
 $$p_{rad}=\frac{2(\alpha-\frac{d}{2})+2s(d-1)}{-(2s-1)(\alpha-\frac{d}{2})+2s(d-1)}.$$

\noindent Moreover, we have the scaling invariant inequality for $u\in X$
$$\| u\|_{L^p(\R^d)}\leq C(\alpha, s, p, d) \, \|\frac{1}{|x|^{\alpha}}\star |u|\|_{L^2(\R^d)}^{1-\theta} \, \|D^s u \|_{L^2(\R^d)}^{\theta},$$
with $p\in(p_{rad},\frac{2d}{d-2s}]$  if $s<\frac{d}{2}$ and  $p\in(p_{rad},\infty)$  if $s\geq \frac{d}{2}$. Here $\theta$ is fixed by the scaling invariance
$$\frac{d}{p}=(1-\theta)((d-\alpha)+\frac{d}{2})+\theta(-s+\frac{d}{2}).$$
\end{thm}
In order to show Theorem \ref{thm:main} we need to prove some preliminary results.

\begin{prop}\label{prop:scalinv}Let $d\geq 1$, $q>1$, $\frac{d}{q}<\alpha<d$, $\delta>0$, then there exists $C>0$ such that
\begin{equation}\label{eq:dalbass1}
\int_{B_R(0)^c} \frac{|u(x)|}{|x|^{\alpha-\frac{d}{q}+\delta}}dx \leq \frac{C}{R^{\delta}} ||\frac{1}{|x|^{\alpha}}\star |u| ||_{L^q(\R^d)}
\end{equation}
\begin{equation}\label{eq:dalbass2}
\int_{B_R(0)} \frac{|u(x)|}{|x|^{\alpha-\frac{d}{q}-\delta}}dx \leq C R^{\delta} ||\frac{1}{|x|^{\alpha}}\star |u| ||_{L^q(\R^d)}.
\end{equation}
\end{prop}
The proposition for $q=2$ has been proved in  \cite{MVS}, we follow the same argument for $q>1$. In order to prove Proposition \ref{prop:scalinv} two crucial lemmas are necessary. The case $q=2$ has  been proved in \cite{MVS} and we follow the same argument.
\begin{lem}\label{lem:important}
Let $d\geq 1$, $q\geq 1$, $0<\alpha<d$, then there exists $C>0$ such that for any $a\in \R^d$
$$\int_0^{\infty} \left(\fint_{B_{\rho}(a)} |u(y)|dy\right)^q\rho^{(d-\alpha) q+d-1}d\rho\leq C||\frac{1}{|x|^{\alpha}}\star |u| ||_{L^q(\R^d)}^q.$$
\end{lem}
\begin{proof}
Let us take $x\in \mathcal{A_{\rho}}=B_{\rho}(a) \setminus B_{\frac{\rho}{2}}(a)$, then
$$\frac{1}{|x|^{\alpha}}\star |u|(x)=\int_{\R^d}\frac{|u(y)|}{|x-y|^{\alpha}}dy\geq $$
$$ \geq \int_{B_{\rho}(a) }\frac{|u(y)|}{|x-y|^{\alpha}}dy \geq C \rho^{d-\alpha}\fint_{B_{\rho}(a)} |u(y)|dy.$$
Thus we obtain for $x\in \mathcal{A_{\rho}}$
$$\left(\frac{1}{|x|^{\alpha}}\star |u|(x)\right)^q\geq C \rho^{(d-\alpha) q}\left(\fint_{B_{\rho}(a)} |u(y)|dy\right)^q$$
and hence
$$\int_{\mathcal{A_{\rho}}}\left(\frac{1}{|x|^{\alpha}}\star |u|(x)\right)^q dx \geq C \rho^{(d-\alpha) q+d}\left(\fint_{B_{\rho}(a)} |u(y)|dy\right)^q. $$
By integration we conclude that
$$\int_{0}^{\infty} \rho^{(d-\alpha) q+d-1}\left(\fint_{B_{\rho}(a)} |u(y)|dy\right)^q d\rho \leq $$ $$ \leq  C  \int_{0}^{\infty} \left(\int_{\mathcal{A_{\rho}}}\left(\frac{1}{|x|^{\alpha}}\star |u|(x)\right)^q dx \right) \frac{d \rho}{\rho}=C ||\frac{1}{|x|^{\alpha}}\star |u| ||_{L^q(\R^d)}^q.$$
\end{proof}
Let us call $W(\rho)=\int_{\rho}^{\infty}w(s)ds$  where $w:(0, \infty)\rightarrow \R$ is a measurable function
such that
\begin{equation}
\int_0^{\infty} |w(\rho)|^{\frac{q}{q-1}}\rho^{\frac{\alpha q+1-d}{q-1}}d\rho<+\infty.
\end{equation}
\begin{lem}\label{lem:important2}
Let $d\geq 1$, $q> 1$, $0<\alpha<d$, then
$$|\int_{\R^d} |u(x)| W(|x|)dx |\lesssim $$ $$ \left(\int_0^{\infty} |w(\rho)|^{\frac{q}{q-1}}\rho^{\frac{\alpha q+1-d}{q-1}}d\rho\right)^{\frac{q-1}{q}}\left(\int_0^{\infty} \left(\fint_{B_{\rho}(a)} |u(y)|dy\right)^q\rho^{\alpha q+d-1}d\rho\right)^{\frac{1}{q}},$$
and hence
\begin{equation}\label{eq:dalbass}
|\int_{\R^d} |u(x)| W(|x|)dx| \leq C ||\frac{1}{|x|^{\alpha}}\star |u| ||_{L^q(\R^d)}.
\end{equation}
\end{lem}
\begin{proof}
We have, thanks to Fubini Theorem,
$$ \int_{\R^d} |u(x)| W(|x|)dx=\int_{\R^d} |u(x)|\left(\int_{|x|}^{\infty}w(\rho)d\rho\right) dx=$$ $$=C \int_0^{\infty}w(\rho)\rho^d \left(\fint_{B_{\rho}(0)} |u(y)|dy\right) d\rho$$
such that by H\"older's inequality we obtain
$$|\int_{\R^d} |u(x)| W(|x|)dx| =C|\int_0^{\infty}w(\rho)\rho^{d-\beta} \left(\fint_{B_{\rho}(0)} |u(y)|dy\right) \rho^{\beta}d\rho | \lesssim $$
$$ \left(\int_0^{\infty} |w(\rho)|^{\frac{q}{q-1}}\rho^{\frac{\alpha q+1-d}{q-1}}d\rho\right)^{\frac{q-1}{q}}\left(\int_0^{\infty} \left(\fint_{B_{\rho}(0)} |u(y)|dy\right)^q\rho^{\alpha q+d-1}d\rho\right)^{\frac{1}{q}},$$
choosing $\beta$ such that $\beta q=(d-\alpha) q+d-1.$
Eq. \eqref{eq:dalbass} comes from Lemma \ref{lem:important}.
\end{proof}
\begin{proof}[Proof of  Proposition \ref{prop:scalinv}]$$$$
If we choose
\begin{equation*}
w(\rho)=\left\{
  \begin{array}{ll}
    0, & \hbox{if $0<\rho<R$;} \\
   \frac{1}{\rho^{\alpha-\frac{d}{q}+1+\delta}}, & \hbox{if $\rho>R$.}
  \end{array}
\right.
\end{equation*}
thanks to Lemma \ref{lem:important2} we get \eqref{eq:dalbass1}. In order to get  \eqref{eq:dalbass2} it is enough to choose
\begin{equation*}
w(\rho)=\left\{
  \begin{array}{ll}
    0, & \hbox{if $\rho>R$;} \\
   \frac{1}{\rho^{\alpha-\frac{d}{q}+1-\delta}}, & \hbox{if $0<\rho<R$.}
  \end{array}
\right.
\end{equation*}

\end{proof}

\begin{lem}\label{lem:decayy2}
Let $d\geq 1$,  $\frac{d}{2}<\alpha<d$  and $||D^s u||_{L^2(\R^d)}=||\frac{1}{|x|^{\alpha}}\star |u| ||_{L^2(\R^d)}=1$, then for any $\delta>0$ such that $0<\delta<d-\alpha$,
$$\int_{\R^d} \frac{|u(x)|}{|x|^{\alpha-\frac{d}{2}+\delta}}dx \leq C(\alpha, s,\delta, d).$$
\end{lem}
\begin{proof}
Let $0<\epsilon<\frac{d}{2}$ be a number to be fixed later. We have
$$\int_{B(0,1)}\frac{|u(x)|}{|x|^{\alpha-\frac{d}{2}+\delta}}dx=\int_{B(0,1)}\frac{|u(x)|}{|x|^{\alpha-\frac{d}{2}+\delta-\epsilon}}\frac{1}{|x|^{\epsilon}}dx\leq  $$ $$ \leq c_{d,\epsilon} \left(\int_{B(0,1)}\frac{|u(x)|^2}{|x|^{2(\alpha-\frac{d}{2}+\delta-\epsilon)}}\right)^{\frac 12},$$
where $c_{d,\epsilon}=\left(\int_{B(0,1)}\frac{1}{|x|^{2\epsilon}}dx\right)^{\frac{1}{2}}$. Now choose $\epsilon=\alpha-\frac{d}{2}+\delta$. Notice that $\epsilon<\frac{d}{2}$ such that
$$\int_{B(0,1)}\frac{|u(x)|}{|x|^{\alpha-\frac{d}{2}+\delta}}dx\leq c_{d,\epsilon} \left(\int_{B(0,1)}|u(x)|^2 dx\right)^{\frac 12}.$$

which implies
$$\int_{B(0,1)}\frac{|u(x)|}{|x|^{\alpha-\frac{d}{2}+\delta}}dx\lesssim 1.$$
On the other hand by Proposition \ref{prop:scalinv}, when $\frac{d}{2}<\alpha<d$
$$\int_{B(0,1)^c} \frac{|u(x)|}{|x|^{\alpha-\frac{d}{2}+\delta}}dx \leq C ||\frac{1}{|x|^{\alpha}}\star |u| ||_{L^2(\R^d)}$$
and hence we obtain the claim.\end{proof}

The next Proposition concerning pointwise decay for  radial functions in $X$ follows the strategy of Theorem 3.1 in \cite{D}. We will decompose the function in high/low frequency part, estimating the high frequency part involving the Sobolev norm while we control the low frequency part involving the Riesz norm.
\begin{prop}\label{prop:tutto}
Let $d \geq 2,$  $u$ be a radial function in $X$ with $s>\frac 12$, $\frac{d}{2}<\alpha<d$,  and
\begin{equation}\label{eq.conJ1}
   ||D^s u||_{L^2(\R^d)}=||\frac{1}{|x|^{\alpha}}\star |u| ||_{L^2(\R^d)}=1.
\end{equation}
Then for any $\sigma$ satisfying
\begin{equation}\label{eq.siim1}
   \frac{2s\left(\frac{d}{2}-1\right) + \left(\frac{d}{2} \right) }{2s+1} < \sigma < \frac{2s(d-1)- (2s-1) \left(\alpha - \frac{d}{2} \right) }{2s+1}
\end{equation}
we have
$$|u(x)|\leq C(\alpha, s,\sigma, d)|x|^{-\sigma}.$$
\end{prop}

\begin{rem}
It is easy to see that the above Proposition  is equivalent to the following statement.

Let $u$ be a radial function in $X$ with $s>\frac 12$, $\frac{d}{2}<\alpha<d$,  and
\begin{equation}\label{eq.conJm1}
   ||D^s u||_{L^2(\R^d)}=||\frac{1}{|x|^{\alpha}}\star |u| ||_{L^2(\R^d)}=1
\end{equation}
then for any $\delta>0$ such that $0<\delta<d-\alpha$,
$$|u(x)|\leq C(\alpha, s,\delta, d)|x|^{-\sigma}$$
with
\begin{equation}\label{e.sig1}
 \sigma=\frac{-(2s-1)(\alpha-\frac{d}{2}+\delta)+2s(d-1)}{2s+1}.
\end{equation}

\end{rem}

\begin{proof}
For any $R>1$ we can take a function $\psi_R(x)=R^{-d}\psi(x/R)$ such  that $\widehat{\psi}(\xi)$ is a radial nonnegative function with support in $|\xi| \leq 2$ and $\widehat{\psi}(\xi)=1$ for $|\xi|\leq 1$ and then we make the  decomposition of $u$ into low and high frequency part as follows
$$u(x)=\psi_R \star u(x)+ h(x)$$
where $\hat h(\xi)=(1-\hat \psi(R|\xi|))\hat u(\xi)$.
 For the high frequency part
we will use
Fourier representation for radial functions in $\R^d$ (identifying the function with its profile)
\begin{equation}\label{eq:rap}
|h(x)|=(2\pi)^{\frac{d}{2}}|x|^{-\frac{d-2}{2}}\int_0^{\infty} J_{\frac{d-2}{2}}(|x|\rho) (1-\psi(R\rho))\hat u (\rho)\rho^{\frac{d}{2}}d\rho
\end{equation}
where $J_{\frac{d-2}{2}}$ is the Bessel function of order $\frac{d-2}{2}.$
Applying the results in  \cite{CO} and \cite{D}, we find
\begin{equation}\label{claim B}
|h(x)|\leq c R^{s-\frac 12}|x|^{-\frac 12(d-1)}||u||_{ \dot H^s(\R^d)}, \  s > \frac{1}{2}.
\end{equation}
Indeed, using the uniform bound
$$ |J_{\frac{d-2}{2}}(\rho)| \lesssim (1+\rho)^{-1/2},   $$
we  get
$$ |h(x)|\lesssim |x|^{-\frac{d-2}{2}}\int_0^{\infty} |(J_{\frac{d-2}{2}})(|x|\rho)| |(1-\psi(R\rho))| |\hat u (\rho)|\rho^{\frac{d}{2}}d\rho \lesssim $$
$$  |x|^{-\frac{d-2}{2}}\left(\int_{1/R}^{\infty} |J_{\frac{d-2}{2}}(|x|\rho)|^2  \frac{ d \rho}{\rho^{2s-1}} \right)^{1/2} \left(\int_0^{\infty} |\hat u (\rho)|^2 \rho^{2s + d-1}d\rho\right)^{1/2} \lesssim$$
$$ |x|^{-\frac{d-2}{2}}R ^{s-1}\left(\int_{1}^{\infty} (1+|x|\rho/R)^{-1}   \frac{ d \rho}{\rho^{2s-1}} \right)^{1/2} \|u\|_{\dot{H}^s(\mathbb{R}^d)} \lesssim $$ $$ \lesssim R^{s-1/2}  |x|^{-\frac{d-1}{2}} \|u\|_{\dot{H}^s(\mathbb{R}^d)}  $$
and this gives \eqref{claim B}.

For low frequency term
$\psi_R \star u(x)$, since $\psi \in \mathcal{S}\left(\mathbb{R}^{d}\right)$, we can take any $\gamma>1$ so that there exists $C>0$ such that
$$
|\psi(x)| \leq C\left(1+|x|^{2}\right)^{-\gamma / 2}.
$$
We shall need the following estimate that can be found also in \cite{DL} and \cite{D}. For sake of completeness we give an alternative proof of the Lemma in the Appendix.

\begin{lem}\label{l.dN1} If
$ b \in (-d+1, 0), \gamma > d-1,$ then for any radially symmetric function $f(|y|)$ we have
\begin{equation}\label{eq.dN1}
   \left| \int_{\mathbb{R}^d} \frac{f(|y|) dy}{(1+|x-y|^2)^{\gamma/2}} \right| \lesssim \frac{1}{|x|^{d-1+b}} \left\| |y|^b f\right\|_{L^1(\mathbb{R}^d)}.
\end{equation}
\end{lem}

Then we  estimate $\psi_R \star u(x)$ as follows,
$$
\begin{aligned}
|\psi_R \star u(x)| & \leq \left|\psi_{R}(x)\right| *|u(x)| \leq C \int_{\mathbb{R}^{d}} \frac{1}{R^{d}} \frac{|u(y)|}{\left(1+\left|\frac{x-y}{R}\right|\right)^{\gamma / 2}} d y \\
& \leq C \int_{\mathbb{R}^{d}} \frac{|u(R z)|}{\left(1+\left|\frac{x}{R}-z\right|^{2} d z\right)^{\gamma / 2}} d z \quad(y=R z) .
\end{aligned}
$$
To this end we plan to apply Lemma \ref{l.dN1} assuming $b=-(\alpha-d/2+\delta)$. To check the assumption of the Lemma we use the inequalities
$$ \alpha - \frac{d}{2}+\delta < \frac{d}{2} \leq d-1$$ for $d \geq 2.$ Applying the Lemma \ref{l.dN1} we deduce
$$
\begin{aligned}
&|\psi_R \star u(x)|\\
& \leq C\left|\frac{x}{R}\right|^{-(d-1+b)} \int_{\mathbb{R}^{d}}|u(R z)|^{}|z|^{b} d z   \\
& \leq C R^{(d-1+b)}|x|^{-(d-1+b)}\int_{\mathbb{R}^{d}}|u(y)|^{}\left|\frac{y}{R}\right|^{b} \frac{d y}{R^{d}} \\
& \leq C R^{-1 }|x|^{-(d-1+b)}\||y|^{b}u\|_{L^{1}(\R^d)}.
\end{aligned}
$$

Therefore, collecting our estimates and using the condition \eqref{eq.conJ1}, we find
$$
\begin{aligned}
|u(x)| & \leq C\left[|x|^{-(d-1) / 2} R^{s-1 / 2}+|x|^{-(d-1+b) } R^{-1 }\||y|^b u\|_{L^{1}(\mathbb{R}^d)}\right].
\end{aligned}
$$

We use Lemma \ref{lem:decayy2} and we get
$$
\begin{aligned}
|u(x)| & \leq C\left[|x|^{-(d-1) / 2} R^{s-1 / 2}+|x|^{-(d-1+b) } R^{-1 }\right].
\end{aligned}
$$

Minimizing in $R$ or equivalently choosing  $R>0$ so that
$$|x|^{-(d-1) / 2} R^{s-1 / 2} = |x|^{-(d-1+b) } R^{-1 },$$
i.e.
$$ R^{s+1/2} = |x|^{-b - (d-1)/2},$$
we find
\begin{equation*}\label{denap1}
 |u(x)|\leq C(d,s,\alpha,\delta)|x|^{-\sigma}.
\end{equation*}
where $\sigma $ is defined in \eqref{e.sig1}.

This completes the proof.

\end{proof}

With all these preliminary results we are now ready to prove Theorem \ref{thm:main}.
\begin{proof}

Let $u\in X$ with $||D^s u||_{L^2(\R^d)}=||\frac{1}{|x|^{\alpha}}\star |u| ||_{L^2(\R^d)}=1$, then by Proposition  \ref{prop:tutto}
$$|u(x)|\leq C(d,s,\alpha,\delta)|x|^{-\sigma}$$
with
$$\sigma=\frac{-(2s-1)(\alpha-\frac{d}{2}+\delta)+2s(d-1)}{2s+1}.$$
 We aim to show that
$p_{rad}<2$, where $p=2$ is the lower endpoint for \eqref{GagliNiren2}. Therefore it sufficies to show that $\int_{|x|>1}|u|^pdx<+\infty$
provided that $u \in X$ and $p_{rad}<p$ (indeed $\int_{|x|\leq 1}|u|^pdx<+\infty$ for all $0<p<2$ by interpolation). \\
We have, thanks to Proposition \ref{prop:tutto} and Lemma \ref{lem:decayy2},
\begin{equation}
\int_{|x|>1}|u||u|^{p-1}dx\lesssim \int_{|x|>1}\frac{|u|}{|x|^{\sigma(p-1)}}dx\lesssim 1
\end{equation}
provided that $\sigma(p-1)>\alpha-\frac{d}{2}$. This condition is equivalent,  $\sigma $ is defined in \eqref{e.sig1} and letting $\delta \rightarrow 0$, to
$$p>\frac{\sigma + \alpha-\frac{d}{2}}{\sigma}=\frac{2(\alpha-\frac{d}{2})+2s(d-1)}{-(2s-1)(\alpha-\frac{d}{2})+2s(d-1)}:=p_{rad}.$$
An elementary computation shows that $p_{rad}<2$ provided that $\frac{d}{2}<\alpha<{d-\frac12}$.

Now consider an arbitrary $v\in X$ and let us call $u=\lambda v(\mu x)$ where the parameters $\lambda, \mu>0$ are chosen such that
 $||D^s u||_{L^2(\R^d)}=||\frac{1}{|x|^{\alpha}}\star |u| ||_{L^2(\R^d)}=1$. By scaling we have
 $$1=||D^s u||_{L^2(\R^d)}=\lambda \mu^{s-\frac{d}{2}}||D^s v||_{L^2(\R^d)}$$
 $$1=||\frac{1}{|x|^{\alpha}}\star |u| ||_{L^2(\R^d)}=\lambda \mu^{\alpha-\frac{3}{2}d}||\frac{1}{|x|^{\alpha}}\star |v| ||_{L^2(\R^d)}$$
 and hence we obtain the relations
 $$\mu=\left(\frac{||D^s v||_{L^2(\R^d)}}{||\frac{1}{|x|^{\alpha}}\star |v| ||_{L^2(\R^d)}}\right)^{\frac{1}{\alpha-s-d}}, \ \ \lambda=\frac{||\frac{1}{|x|^{\alpha}}\star |v| ||_{L^2(\R^d)}^{\frac{s-\frac{d}{2}}{\alpha-d-s}}}{||D^s v||_{L^2(\R^d)}^{\frac{\alpha-\frac{3}{2}d}{\alpha-d-s}}}.$$
 By the previous estimates we have
 $$||u||_{L^p(\R^d)}=\lambda \mu^{-\frac{d}{p}}||v||_{L^p(\R^d)}\lesssim 1$$
 which implies
 $$||v||_{L^p(\R^d)}\lesssim \lambda^{-1}\mu ^{\frac{d}{p}} = ||D^s v||_{L^2(\R^d)}^{\theta} \|\frac{1}{|x|^{\alpha}}\star |v|\|_{L^2(\R^d)}^{1-\theta},$$
 where
 $$\theta = \frac{d^2-2\alpha p+3dp-2ds}{2p(d+s-\alpha)}, \ \ 1-\theta = \frac{(2s-d)(d+p)}{2p(d+s-\alpha)}.$$
 It is easy to see that
 $\theta$ is fixed by the scaling invariance
$$\frac{d}{p}=(1-\theta)((d-\alpha)+\frac{d}{2})+\theta(-s+\frac{d}{2}).$$
\end{proof}

\section{Proof of Theorem \ref{thm:super}}

Our goal is to represent $\varphi$ in the form  $\varphi = \frac{1}{|x|^{\alpha}}\star u = c D^{-r} u, $ with $\alpha = d-r, $
$ c = \frac{\pi^{d / 2} \Gamma((d-\alpha) / 2)}{ \Gamma(\alpha / 2)}$
 and apply Theorem \ref{thm:main}.
Therefore, we choose (modulo constant)   $u = D^r \varphi.$

Then the estimate of Theorem \ref{thm:main} gives
$$ \|D^r \varphi\|_{L^p(\R^d)} = \| u\|_{L^p(\R^d)}\lesssim  \, \|\frac{1}{|x|^{\alpha}}\star |u|\|_{L^2(\R^d)}^{1-\theta} \, \|D^s u \|_{L^2(\R^d)}^{\theta} =$$
$$ =   \left\| D^{-r} |D^{r} \varphi| \right\|_{L^2(\R^d)}^{1-\theta} \|D^S \varphi|\|_{L^2(\R^d)}^{\theta}.$$
By the assumption
\begin{equation}\label{eq.pos1}
   D^{r} \varphi( x) \geq 0
\end{equation}
for almost every $x \in \mathbb{R}^d,$ then we deduce
$$   \left\| D^{-r} |D^{r} \varphi| \right\|_{L^2(\R^d)}^{1-\theta} \|D^S \varphi|\|_{L^2(\R^d)}^{\theta}=   \left\| D^{-r} D^{r} \varphi \right\|_{L^2(\R^d)}^{1-\theta} \|D^S \varphi|\|_{L^2(\R^d)}^{\theta}$$ and we obtain \eqref{GagliNiren5}. The lower endpoint $p_0$ is hence nothing but $p_{rad}$ of Theorem \ref{thm:main} substituting $\alpha$ with $d-r$ and $s$ with $S-r$.
The condition $\frac 12 <r<\min(\frac{d}{2}, S-\frac 12)$ is  equivalent to  the conditions $\frac{d}{2}<\alpha<d-\frac 12$, $s>\frac 12$ of Theorem \ref{thm:main}. All these estimates
remain valid if we consider    $D^{r} \varphi( x) \leq 0$, i.e if $\varphi \in H^{s,r}_{rad, -}(\R^d)$. Indeed if $\varphi \in H^{s,r}_{rad, -}(\R^d)$
$$   \left\| D^{-r} |D^{r} \varphi| \right\|_{L^2(\R^d)}^{1-\theta} \|D^S \varphi|\|_{L^2(\R^d)}^{\theta}=  $$ $$ =  \left\| -D^{-r} D^{r} \varphi \right\|_{L^2(\R^d)}^{1-\theta} \|D^S \varphi|\|_{L^2(\R^d)}^{\theta}=\left\|  \varphi \right\|_{L^2(\R^d)}^{1-\theta} \|D^S \varphi|\|_{L^2(\R^d)}^{\theta}.$$
\section{Proof of Theorem \ref{thm:main2}}

We prove that under  the assumption of Theorem \ref{thm:main2}, the embedding
$$H^{S,r_0}_{rad,+}(\R^d)\subset \subset \dot H^{r_0}_{rad}(\R^d)
$$
is compact. As a byproduct  the embedding
\begin{equation}\label{eq:intcomp}
H^{S,r_0}_{rad,+}(\R^d)\subset \subset \dot H^{r}_{rad}(\R^d)
\end{equation}
is compact for any $0<r<S$. The embedding \eqref{eq:intcomp} follows noticing that if  $\varphi_n$ converges weakly to some $\varphi$ in $H^{S }_{rad}(\R^d)$ then  $\varphi_n$ converges weakly to the same $\varphi$ in $H^{ r_0 }_{rad}(\R^d)$. Now if we prove that (taking a subsequence)
\begin{equation}\label{eq.cco1}
    \|D^{r_0} (\varphi_n-\varphi)\|_{L^2}=o(1)
\end{equation}
as $n \to \infty,$ then
 by the following interpolation inequalities
$$\|D^r (\varphi_n-\varphi)\|_{L^2}\lesssim\|D^{r_0} (\varphi_n-\varphi)\|_{L^2}^{1-\frac{r-r_0}{S-r_0}} \, \|D^{S} (\varphi_n-\varphi)\|_{L^2}^{\frac{r-r_0}{S-r_0}}=o(1)$$
if $0<r_0<r<S$ and
$$\|D^r (\varphi_n-\varphi)\|_{L^2}\lesssim\| (\varphi_n-\varphi)\|_{L^2}^{1-\frac{r}{r_0}} \, \|D^{r_0} (\varphi_n-\varphi)\|_{L^2}^{\frac{r}{r_0}}=o(1)$$
if  $0<r<r_0$, we get \eqref{eq:intcomp}.

To prove \eqref{eq.cco1} we recall that  $(\varphi_n)_{n\in\mathbb N}$ is a bounded sequence in $H^{S,r_0}_{rad,+}(\R^d)$ and we can assume that $\varphi_n$ converges weakly to some $\varphi$ in $H^{S}(\R^d)$. To simplify the notation we will use $r$ instead of $r_0$ in the proof of \eqref{eq.cco1}. We choose a bump function  $\theta\in C_0^\infty(\R^d)$, such that  $\theta = 1$ on $B_1$ and $\theta = 0$ in $\R^d \setminus B_{2 }$  and for any $ \rho > 1$ we define $\theta_\rho(x) = \theta(x/\rho).$ Clearly the   multiplication by $\theta_{\rho} \in \mathcal S(\R^d)$ is a continuous mapping $H^{S} (\R^d)\rightarrow H^{S} (\R^d)$.
Now setting $v_n=\theta_\rho \varphi_n$ and $v=\theta_\rho \varphi$ we aim to show that
\begin{equation}\label{eq.smth1}
   \lim_{n \to \infty} \|D^{r}( v_n-v)\| _{L^{2}(\R^d)}^2 = \lim_{n \to \infty} \|D^{r}(\theta_{\rho}( \varphi_n-\varphi))\| _{L^{2}(\R^d)}^2 = 0.
\end{equation}
for any $r \in [0, S).$

Indeed, by Plancharel's identity we have
\begin{equation*}
\|D^{r}( v_n-v)\| _{L^{2}(\R^d)}^2=\underbrace{\int_{| \xi|\leq R} |\xi|^{2r}| \widehat v_n(\xi)-\widehat v(\xi)|^2 d\xi}_{=I}+\underbrace{\int_{ |\xi| > R}|\xi|^{2r}| \widehat v_n(\xi)-\widehat v(\xi)|^2 d\xi}_{=II}.
\end{equation*}
Clearly $$II\leq \frac{1}{R^{2( S-r)}}\int_{ |\xi| > R}|\xi|^{2 S}| \widehat v_n(\xi)-\widehat v(\xi)|^2 d\xi$$
and then we can choose $R>0$ such that $II\leq \frac{\epsilon}{2}$.
\\
Since $e^{-2\pi i x\cdot \xi}\in L^2_x(B_{2 \rho})$, by weak convergence in $L^2(B_{2 \rho})$ we have $\widehat v_n(\xi) \rightarrow \widehat v (\xi)$ almost everywhere. Notice that $||\widehat v_n ||_{L^{\infty}}\leq || v_n|| _{L^1(B_{2 \rho})}\leq
\mu(B_{2 \rho})^{\frac 12} || v_n||_{L^2(B_{2 \rho})}\leq \mu(B_{2 \rho})^{\frac 12}||v_n||_{H^{ S}(\R^d)}$ and hence  $| \widehat v_n(\xi)-\widehat v(\xi)|^2$ is estimated by a uniform constant so that by Lebesgue's dominated convergence theorem
$$I=\int_{| \xi|\leq R} |\xi|^{2r}| \widehat v_n(\xi)-\widehat v(\xi)|^2 d\xi<\frac{\epsilon}{2},$$
for $n$ sufficiently large. This proves \eqref{eq.smth1}.

Our next step is to
 show that for a given $\varepsilon >0$ one can find  $\rho_0=\rho_0(\varepsilon)$ sufficiently large and $n_0(\varepsilon)$ sufficiently large so that
\begin{equation}\label{eq.mestr01}
  \|\theta_{\rho} D^{r}( \varphi_n-\varphi)\| _{L^{2}(\R^d)}^2 \leq \frac{\varepsilon}{2}
\end{equation}
for $ n\geq n_0, \rho\geq \rho_0$ and any $ r \in [0,S).$

We  consider first  the case $0\leq r \leq 2, r<S.$ The cases $r=0$ and  $r=2$ are trivial, for this we assume  $0 < r  <  \min (2,S).$
We  shall use the following statement (see Corollary 1.1 in \cite{FGO}).
\begin{prop}
\label{Corollary:1.1}
Let $p,p_1,p_2$ satisfy $1 < p, p_1, p_2 < \infty$ and $1/p = 1/p_1 + 1/p_2$.
Let $r,r_1,r_2$ satisfy $0 \leq r_1, r_2 \leq 1$, and $r = r_1 + r_2$.
Then the following bilinear estimate
	\[
	\|\underbrace{ D^r(fg) - f D^r g}_{= [D^r,f]g} - g D^r f \|_{L^{p}}
	\leq C \|D^{r_1} f\|_{L^{p_1}} \| D^{r_2} g\|_{L^{p_2}}
	\]
holds for all $f,g \in \mathcal S$.
\end{prop}

By a density argument the statement holds for $f,g \in H^S(\mathbb{R}^d).$
We choose
$f = \theta_\rho,$ $g=\varphi_n-\varphi$   and $r_1=r_2 = r/2$ and therefore we aim to use  \eqref{eq.smth1} and prove that
\begin{equation}\label{eq.tre1}
\begin{aligned}
 & \|[\theta_\rho, D^{r}]( \varphi_n-\varphi)\| _{L^{2}(\R^d)}  \leq \\ & \leq O(\rho^{-r}) \|\varphi_n-\varphi\|_{L^{2}(\R^d)} + O(\rho^{-r/4})\| \varphi_n-\varphi\| _{H^{r}(\R^d)}.
  \end{aligned}
\end{equation}

Indeed from the Proposition \ref{Corollary:1.1} we have
$$ \|[\theta_\rho, D^{r}]( \varphi_n-\varphi)\| _{L^{2}(\R^d)} \lesssim  \underbrace{\|D^{r} \theta_\rho\| _{L^{\infty}(\R^d)}\| \varphi_n-\varphi\| _{L^{2}(\R^d)}}_{= O(\rho^{-r})}+ $$
$$ + \| D^{r/2} \theta_\rho\| _{L^{p_1}(\R^d)} \| D^{r/2}( \varphi_n-\varphi))\| _{L^{p_2}(\R^d)}.$$
It is easy to check the estimate
$$ \| D^{r/2} \theta_\rho\| _{L^{p_1}(\R^d)}  = O(\rho^{-r/4}),$$
as $\rho \to \infty,$ and this is obviously fulfilled if
$\frac{d}{p_1} < \frac{r}{4}$. To control
$ \| D^{r/2}( \varphi_n-\varphi))\| _{L^{p_2}(\R^d)} $
we use Sobolev inequality
$$ \| D^{r/2}( \varphi_n-\varphi))\| _{L^{p_2}(\R^d)} \lesssim  \| \varphi_n-\varphi\| _{H^r(\R^d)} $$
so we need
$$ \frac{1}{p_2} > \frac{1}{2}- \frac{r-r/2}{d}.$$
Summing up we have the following restrictions for $1/p_1, 1/p_2$
\begin{equation}\label{eq.sim1}
\begin{aligned}
&\frac{1}{p_1}+\frac{1}{p_2} = \frac{1}{2}\\
&\frac{1}{p_1} < \frac{r}{4d}, \ \frac{1}{p_2} > \frac{1}{2}- \frac{r-r/2}{d}.
\end{aligned}
\end{equation}
Choosing
$ p_2= 2+ \kappa, p_1= 2(2+\kappa)/\kappa$ with $\kappa >0$ sufficiently small we see that \eqref{eq.sim1} is nonempty.
Now notice that
\begin{equation}\label{eq.quellocheserv}  \|\theta_{\rho} D^{r}( \varphi_n-\varphi)\| _{L^{2}(\R^d)}\leq   \|D^{r}( \theta_{\rho}(\varphi_n-\varphi)\| _{L^{2}(\R^d)}+ \|[\theta_\rho, D^{r}]( \varphi_n-\varphi)\| _{L^{2}(\R^d)}
\end{equation}

and  we conclude that \eqref{eq.mestr01} is true for $0 \leq r<\min(2, S)$ thanks to  \eqref{eq.smth1} and \eqref{eq.tre1}.

Now we consider the case
$2 \leq r<S.$ We have
$ D^{r} = D^{r_1}(-\Delta)^\ell,$ where $\ell \geq 1$ is integer and $ 0 < r_1 < 2.$ Then the commutator relation
$$ [A,BC]= [A,B]C+B[A,C]$$ implies
$$ [\theta_\rho, D^{r}] = [\theta_\rho,  D^{r_1} ](-\Delta)^\ell+ D^{r_1}[\theta_\rho,(-\Delta)^\ell ].$$
In fact, we have the relation
$$ \theta_\rho D^{r} ( \varphi_n-\varphi) = [\theta_\rho,  D^{r_1} ] ( (-\Delta)^\ell(\varphi_n-\varphi)) +  D^{r_1}[\theta_\rho,(-\Delta)^\ell ]( \varphi_n-\varphi)$$
and we use \eqref{eq.tre1} so that
$$\|[\theta_\rho, D^{r_1}](-\Delta)^\ell( \varphi_n-\varphi)\| _{L^{2}(\R^d)} \leq $$ $$ \leq O(\rho^{-r_1}) \|(-\Delta)^\ell (\varphi_n-\varphi)\|_{L^{2}(\R^d)} + $$ $$ + O(\rho^{-r_1/4})\| D^{r_1+2\ell}( \varphi_n-\varphi)\| _{L^{2}(\R^d)} = o(1)$$
for $\rho \to \infty.$

The term $$ D^{r_1}[\theta_\rho,(-\Delta)^\ell ]( \varphi_n-\varphi)$$ can be evaluated pointwise  via the classical Leibnitz rule and then via the fractional Leibnitz rule  as follows
$$  \|D^{r_1}[\theta_\rho,(-\Delta)^\ell ]( \varphi_n-\varphi)\|_{L^{2}(\R^d)} \lesssim $$ $$ \lesssim \sum_{1 \leq |\alpha|, |\alpha| +|\beta|=2\ell}  \|D^{r_1}(\partial_x^\alpha \theta_\rho) \partial_x^\beta ( \varphi_n-\varphi)\|_{L^{2}(\R^d)} \lesssim O(\rho^{-1}) \|\varphi_n-\varphi\|_{H^r(\mathbb{R}^d)}.$$

Summing up, we conclude that \eqref{eq.mestr01} holds in case $r \in [0,S).$

To conclude that the embedding is compact it remains to show that also $\|D^{r}( \varphi_n-\varphi)\| _{L^{2}(B_{\rho}^c)}^2\leq \epsilon$.
To this purpose we first use the pointwise decay in terms of  homogeneous Sobolev norm, see \cite{CO}. Given $r$ there exists $0<\delta<\frac{d-1}{2}$ with $r+\frac{1}{2}+\delta<S$ such that
\begin{equation}\label{eq:decaycc}
|D^{r}( \varphi_n-\varphi)(x)|\leq \frac{C}{|x|^{\gamma}}||\varphi_n-\varphi||_{\dot H^{r+\frac{1}{2}+\delta}(\R^d)}\lesssim  \frac{C}{|x|^{\gamma}}
\end{equation}
with $\gamma=\frac{d-1}{2}-\delta$.
Secondly we use  that $p_0<2$, i.e. that $p=2$ is non endpoint. By Theorem \ref{thm:super} there exists $\delta_0>0$ sufficiently small such that $D^{r} \varphi_n$ is uniformly bounded  in $L^{2-\delta_0}(\R^d)$ and the same holds hence for   $D^{r} \varphi$  and $D^{r}( \varphi_n-\varphi)$. As a consequence we have
$$\|D^{r}( \varphi_n-\varphi)\| _{L^{2}(B_{\rho}^c)}^2=\int_{B_{\rho}^c} |D^{r}( \varphi_n-\varphi)|^{\delta_0}|D^{r}( \varphi_n-\varphi)|^{2-\delta_0}dx\leq $$ $$ \leq \frac{C}{\rho^{\gamma}}^{\delta_0}\|D^{r}( \varphi_n-\varphi)\| _{L^{2-\delta_0}(B_{\rho}^c)}^{2-\delta}$$
with
$$\|D^{r}( \varphi_n-\varphi)\| _{L^{2-\delta_0}(B_{\rho}^c)}\leq \|D^{r}( \varphi_n-\varphi)\| _{L^{2-\delta_0}(\R^d)}=O(1).$$
This proves that $\|D^{r}( \varphi_n-\varphi)\| _{L^{2}(B_{\rho})}^2\lesssim \epsilon$ and hence that the embedding is compact.

\section{Proof of Theorem \ref{thm:main3}}
For easier reference we state the following.

\begin{lem}[\emph{pqr} Lemma \cite{FLL}]\label{Lieb}
Let $1\leq p<q<r\leq\infty$ and let $\alpha, \beta, \gamma>0$. Then there are constants $\eta,c>0$ such that for any measurable function $f\in L^p(X)\cap L^r(X)$, $X$ a measure space, with
\begin{equation*}
\|f\|_{L^p}^p\leq \alpha, \quad
\|f\|_{L^q}^q\geq \beta, \quad
\|f\|_{L^r}^r\leq \gamma, \quad
\end{equation*}
one has (with $|\cdot|$ denoting the underlying measure on $X$)
\begin{equation}\label{statement}
\left|\{ x \in X :\ |f(x)|>\eta \}\right| \geq c \,.
\end{equation}
\end{lem}
\begin{lem}[Compactness up to translations in $\dot H^{s}$  \cite{BFV}]\label{LiebIntro}
Let $s>0$, $1<p<\infty$ and $u_n\in \dot H^s(\R^d)\cap L^p(\R^d)$ be a sequence with
\begin{equation}\label{hyp1}
\sup_n \left( \|u_n\|_{\dot H^{s}(\R^d)} +\|u_n\|_{L^p(\R^d)} \right) <\infty
\end{equation}
and, for some $\eta>0$, (with $|\cdot |$ denoting Lebesgue measure)
\begin{equation}\label{hyp2}
\inf_n \left|\{ |u_n|>\eta \}\right|>0 \,.
\end{equation}
Then there is a sequence $(x_n)\subset\R^d$ such that a subsequence of $u_n(\cdot+ x_n)$ has a weak limit $u\not\equiv 0$ in $\dot H^s(\R^d)\cap L^p(\R^d)$.
\end{lem}

The strategy to prove  Theorem \ref{thm:main3} follows the one developed in \cite{BFV}. First we aim to show that the maximum of
$$W(\varphi)=\frac{\|D^r \varphi\|_{L^2(\R^d)}}{ \|\varphi\|_{L^2(\R^d)}^{1-\frac{r}{S}} \, \|D^{S} \varphi\|_{L^2(\R^d)}^{\frac{r}{S}}}  \ \ \ \varphi \in  H^{S,r}_{rad,+}(\R^d),$$
is achieved in   $H^{S,r}_{rad,+}(\R^d)$ . Let us consider a maximizing sequence $\varphi_n$. Since $W$ is invariant under
homogeneity $\varphi(x) \mapsto \lambda \varphi(x)$ and scaling
$\varphi \mapsto \varphi(\lambda x)$ for any $\lambda>0$, we can choose a maximizing sequence $\varphi_n$ such that
\begin{equation}\label{0.0}
\|D^r \varphi_n\|_{L^2(\R^d)}= C_{rad,+}(r,S,2, d)+o(1)
\end{equation}
and
\begin{equation}\label{0.1}
\|\varphi_n\|_{L^2(\R^d)}=\|D^{S} \varphi_n\|_{L^2(\R^d)}=1 \,.
\end{equation}
The key observation is that, since we are looking at a non-endpoint case (i.e. $p_0<2$), there exists $\epsilon>0$ such that from inequality \eqref{GagliNiren5}  we infer that
\begin{equation}\label{inequ}
\sup_n \max\left\{\|D^r \varphi_n\|_{L^{2-\epsilon}(\R^d)},
\|D^r \varphi_n\|_{L^{2+\epsilon}(\R^d)}\right\}<\infty \,.
\end{equation}
The $pqr$-lemma (Lemma \ref{Lieb}) now implies that
\begin{equation}\label{superlevel}
\inf_n \left|\{ |D^r \varphi_n|>\eta \}\right| >0.
\end{equation}
Next, we apply the compactness modulo translations lemma (Lemma \ref{LiebIntro}) to the sequence $(D^r \varphi_n)$. This sequence is bounded in $\dot H^{S-r}$ by \eqref{0.1}, and \eqref{hyp1} and \eqref{hyp2} are satisfied by \eqref{0.0} and \eqref{superlevel}. Thus possibly after passing to a subsequence, we have $D^r \varphi_n \rightharpoonup \psi\not\equiv 0$ in $ H^{S-r}(\R^d)$. By the fact the embedding is compact we deduce that $\varphi_n(x) \rightarrow  \psi\not\equiv 0$ in $ \dot H^{r}(\R^d)$ and hence $ \psi$ is a maximizer for $W$.\\
Now we conclude showing that $C_{rad,+}(r,S,2, d)<1$.\\
Indeed if the best constant is  $C_{rad,+}(r,S,2, d)=1$, the maximizer $\psi$ achieves the equality in H\"older's inequality, which means
\begin{equation}\label{bellisssimo}
\begin{aligned}
& \int_{\R^d} |\xi|^{2r}|\widehat \psi|^2d\xi=\int_{\R^d} |\widehat \psi|^{2-\frac{2r}{S}} |\xi|^{2r}|\widehat \psi|^{\frac{2r}{S}}d\xi= \\ & =\left(\int_{\R^d} |\widehat \psi|^2d\xi\right)^{1-\frac{r}{S}}\left(\int_{\R^d} |\xi|^{2S}|\widehat \psi|^2d\xi\right)^{\frac{r}{S}},
\end{aligned}
\end{equation}
where we used as conjugated exponents $\frac{S}{S-r}$ and $\frac{S}{r}$. Now we recall that if $f\in L^p(\R^d)$ and $g\in L^q(\R^d)$ with $p$ and $q$ conjugated exponents achieve the equality
in H\"older's inequality then $|f|^p$ and $|g|^q$ shall be linearly dependent, i.e.  for a suitable $\mu,  |f|^p=\mu |g|^q$ almost everywhere.
Therefore, calling $f=|\widehat \psi|^{2-\frac{2r}{S}}$ and $g= |\xi|^{2r}|\widehat \psi|^{\frac{2r}{S}}$,  the maximizer $\psi$ should satisfy $ |\widehat \psi|^{2}=\mu |\xi|^{2S}|\widehat \psi|^2 $ for a suitable $\mu$ which drives to the contradiction $\widehat \psi=0.$

\section{Appendix.}

The statement of Lemma \ref{l.dN1} can be found in \cite{DL}. Somehow, due to the fact that in the original paper the proof of Lemma \ref{l.dN1} is not easy readable, being a part of a more general statement, we give an alternative short proof.
\begin{proof}[Proof of Lemma \ref{l.dN1}]
We divide the integration domain in two subdomains:
$$ \Omega= \{|x| < |y|/2 \} \cup \{|x| > 2|y|\} $$
and its complementary set $ \Omega^c.$ In $\Omega$ we use
$$ |x-y| \geq \frac{\max(|x|,|y|)}{2} $$ and via
$$ (1+ |x-y| ^2)^{(d-1)/2} \gtrsim (1+ (\max(|x|,|y|)) ^2)^{(d-1)/2} \geq $$ $$ \geq  \max(|x|,|y|))^{(d-1)}\gtrsim |x|^{(d-1+b)} |y|^{-b} $$
with $d-1+b>0, -b>0$ we deduce
$$ \frac{1}{(1+|x-y|^2)^{\gamma/2}} = \frac{1}{(1+|x-y|^2)^{(d-1)/2}} \frac{1}{(1+|x-y|^2)^{(\gamma-d+1)/2}}$$ $$  \lesssim \frac{1}{|x|^{d-1+b}} |y|^b \frac{1}{(1+|x-y|^2)^{(\gamma-d+1)/2}} \leq  \frac{1}{|x|^{d-1+b}} |y|^b .$$
These estimates imply
\begin{equation}\label{eq.dN2}
   \left| \int_{\Omega} \frac{f(y) dy}{(1+|x-y|^2)^{\gamma/2}} \right| \lesssim \frac{1}{|x|^{d-1+b}} \left\| |y|^b f\right\|_{L^1(\mathbb{R}^d)}.
\end{equation}
For the complementary domain $\Omega^c$ we use spherical coordinates $x=r\theta, y =\rho \omega,$ where $r=|x|, \rho=|y|.$ We have to estimate
$$ \int_{\Omega^c} \frac{f(y) dy}{(1+|x-y|^2)^{\gamma/2}} = \int_{r/2}^{2r} K(r,\rho) f(\rho) \rho^{d-1} d\rho, $$
where
\begin{equation}\label{eq.ke2}
     K(r,\rho) = K_{\theta,\gamma}(r,\rho) = \int_{\mathbb{S}^{d-1}} (1+|r\theta-\rho\omega|^2)^{-\gamma/2} d \omega.
\end{equation}
To get the desired estimate
\begin{equation}\label{eq.dN3}
   \left| \int_{\Omega^c} \frac{f(y) dy}{(1+|x-y|^2)^{\gamma/2}} \right| \lesssim \frac{1}{|x|^{d-1+b}} \left\| |y|^b f\right\|_{L^1(\mathbb{R}^d)}
\end{equation}
it is sufficient to check the pointwise estimate
\begin{equation}\label{eq.kes1}
   K(r,\rho) \lesssim  r^{-(d-1+b)} \rho^{b}  \sim r^{-(d-1)} \ \ \mbox{for  $ r/2 \leq \rho \leq 2r$.}
\end{equation}

To deduce this pointwise estimate of the kernel $K$ we note first that $K$ does not depend on $\theta$ so we can take
$\theta=e_d=(0,\cdots,0,1)$ and $\omega = ( \omega^\prime \sin \varphi, \cos \varphi),$ $\omega^\prime \in \mathbb{S}^{d-2}$ and get
$$K(r,\rho) = c\int_0^\pi \frac{\sin ^{d-2} \varphi d \varphi}{(1+r^2+\rho^2-2r\rho\cos \varphi)^{\gamma/2}}.$$
Using the relation
$$ (1+r^2+\rho^2-2r\rho\cos \varphi) = 1+(r-\rho)^2 + r\rho \sin^2 (\varphi/2),$$
we can use the
$$ (1+r^2+\rho^2-2r\rho\cos \varphi) \gtrsim r\rho \sim r^2 $$
when $\rho \sim r$ and  $\varphi$ is not close to $0,$ say $\varphi \in (\pi/4, \pi).$
Then we get
$$ \int_{\pi/4}^\pi \frac{\sin ^{d-2} \varphi d \varphi}{(1+r^2+\rho^2-2r\rho\cos \varphi)^{\gamma/2}} \lesssim \int_{\pi/4}^\pi r^{-\gamma} d\varphi \lesssim r^{-\gamma}  \leq r^{-d+1} .$$

For $\varphi$ close to $0$, say $\varphi \leq \pi/4$ we use
$$ \frac{\sin ^{d-2} \varphi }{(1+r^2+\rho^2-2r\rho\cos \varphi)^{\gamma/2}} \lesssim \frac{\varphi^{d-2}}{(1+r\rho \varphi^2)^{\gamma/2}}.$$
In this way, making the change of variable $r\varphi=\eta$ we get
$$ \int_0^{\pi/4}  \frac{\varphi^{d-2} d\varphi}{(1+r\rho \varphi^2)^{\gamma/2}} \lesssim \int_0^{\infty}  \frac{\varphi^{d-2} d\varphi}{(1+r^2 \varphi^2)^{\gamma/2}}  \leq $$ $$ \leq  r^{-d+1}  \int_0^{\infty}  \frac{\eta^{d-2} d\eta}{(1+ \eta^2)^{\gamma/2}} \lesssim r^{-d+1}   $$
in view of $\rho \sim r$ and $\gamma > d-1.$
Taking together the above estimates of the integrals over $(0,\pi/4)$
 and $(\pi/4,\pi)$, we arrive at \eqref{eq.kes1}.

 This completes the proof of the Lemma.

 \end{proof}


\begin{thebibliography}{99}
\bibitem{BCD} H. Bahouri, J.-Y. Chemin, R. Danchin, \textit{Fourier Analysis and Nonlinear Partial Differential Equations}, Springer, 2011.
\bibitem{BFV} {J.Bellazzini, R.L. Frank, N. Visciglia, }{\sl Maximizers for Gagliardo-Nirenberg inequalities and related non-local problems},  Math. Annalen  (2014), no. 3-4, 653 -- 673.
\bibitem{BGMMV} J. Bellazzini, M. Ghimenti, C. Mercuri, V. Moroz, J. Van Schaftingen, {\sl Sharp Gagliardo-Nirenberg inequalities in fractional Coulomb-Sobolev spaces,} Trans. Amer. Math. Soc.  370,  11, 8285 -- 8310 (2018)
\bibitem{BGO} J. Bellazzini, M. Ghimenti, and T. Ozawa, \emph{Sharp lower bounds for Coulomb energy}, Math. Res. Lett. 23 (2016), no. 3, 621--632


\bibitem{BM01}
H. Brezis,  P. Mironescu, \textit{ Gagliardo-Nirenberg, composition and products in fractional Sobolev spaces. Dedicated to the memory of Tosio Kato.} J. Evol. Equ. 1 (2001), no. 4, 387 -- 404


\bibitem{BM19} H. Brezis and P. Mironescu, {\sl  Where Sobolev interacts with Gagliardo-Nirenberg,} J. Funct. Anal. 277 (2019), no. 8, 2839 -- 2864.


\bibitem{CO}Y. Cho, T. Ozawa, \emph{Sobolev inequality with symmetry},  Commun. Contemp. Math. 11 (2009), no. 3, 355--365
\bibitem{D}  P.L. De N\'{a}poli  \emph{Symmetry breaking for an elliptic equation involving the Fractional Laplacian, } Differential Integral Equations 31 (1/2) 75 - 94, January/February (2018).
\bibitem{FLL} J. Fr\"ohlich,  E. H. Lieb, M. Loss, \emph{Stability
of Coulomb systems with magnetic fields. I. The one-electron atom.},
Comm. Math. Phys. 104 (1986), no. 2, 251--270.

\bibitem{FGO}
K. Fujiwara, V.Georgiev, T. Ozawa,
\emph{ Higher order fractional Leibniz rule.}
J. Fourier Anal. Appl. 24 (2018), no. 3, 650 -- 665.

\bibitem{HMOW11} H. Hajaiej,L. Molinet, T. Ozawa, and B.  Wang,  \emph{ Necessary and sufficient conditions for the fractional Gagliardo-Nirenberg inequalities and applications to Navier-Stokes and generalized boson equations,} Harmonic analysis and nonlinear partial differential equations, 159 -- 175, RIMS K\^oky\^uroku Bessatsu, B26, Res. Inst. Math. Sci. (RIMS), Kyoto, 2011


\bibitem{DL} P. D'Ancona, R. Luc$\grave{a}$, \emph{Stein-Weiss and Caffarelli-Kohn-Nirenberg inequalities with angular
integrability},  J. Math. Anal. Appl., 388(2): 1061--1079, (2012)
\bibitem{MVS} C. Mercuri, V. Moroz, J. Van Schaftingen  \emph{Groundstates and radial solutions to nonlinear Schr\"odinger-Poisson-Slater  equations at the critical frequency}, Calc. Var. 55, 146 (2016)
\bibitem{R} D. Ruiz, \emph{On the Schr\"odinger-Poisson-Slater system: behavior of minimizers, radial and nonradial cases},  Arch. Ration. Mech. Anal. 198 (2010), no. 1, 349--368
%
\bibitem{SS00}
W. Sickel, L. Skrzypczak, \emph{Radial subspaces of Besov and Lizorkin-Triebel classes: extended Strauss lemma and compactness of embeddings},  J. Fourier Anal. Appl. 6 (2000), no. 6, 639 -- 662.

\bibitem{SS12} W. Sickel,  L. Skrzypczak, \emph{ On the Interplay of Regularity and Decay in Case
of Radial Functions II. Homogeneous Spaces},
J Fourier Anal Appl (2012) 18:548  -- 582
DOI 10.1007/s00041-011-9205-2

\bibitem{S77} W. A. Strauss. \emph{Existence of Solitary Waves in Higher Dimensions}. Comm. Math. Phys. 55 (1977), 149--162.




\end{thebibliography}
\end{document}